\documentclass[11pt]{amsart}
\usepackage[letterpaper,margin=1.1in]{geometry}
\usepackage{amsmath,amsthm,amssymb,amsfonts,graphicx,mathrsfs,bm,amscd,latexsym,caption}
\usepackage{color,enumitem,hyperref}
\usepackage{blindtext}
\usepackage{textcomp}
\usepackage[utf8]{inputenc}
\usepackage[all]{xy}
\usepackage{wasysym}
\newtheorem{theorem}{Theorem}[section]
\newtheorem{proposition}[theorem]{Proposition}
\newtheorem{lemma}[theorem]{Lemma}

\theoremstyle{definition}
\newtheorem{rem}[theorem]{Remark}
\newtheorem{definition}[theorem]{Definition}
\newtheorem{ex}[theorem]{Example}
\newtheorem{question}[theorem]{Question}

\def\R{\mathbb{R}}

\def\Q{\mathbb{Q}}
\def\N{\mathbb{N}}

\DeclareMathOperator{\wtan}{Tan}

\newcommand{\whocare}[1]{}

\DeclareMathOperator{\dimh}{dim_H}

\DeclareMathOperator{\dist}{dist}

\DeclareMathOperator{\Lip}{Lip}

\DeclareMathOperator{\exc}{excess}

\numberwithin{equation}{section}
\title{pseudotangents to lipschitz curves}
\author{Eve Shaw}
\date{\today}

\address{Department of Mathematics\\ The University of Illinois\\ Urbana, IL, 61801}
\email{emshaw@illinois.edu}
\begin{document}
\subjclass[2020]{Primary 28A75; Secondary 51F30, 53A04}
\keywords{Lipschitz curves, tangents}
\begin{abstract}
In this paper, we extend the result of \cite[Appendix A]{QSS} by showing that the set on which every pseudotangent is obtained on a Lipschitz curve can be any compact, uniformly disconnected set in Euclidean space which admits any Lipschitz capture. We do not obtain a characterization of such sets however, indeed we leave open the very strong question of whether or not a Lipschitz curve can obtain every pseudotangent at every point.
\end{abstract}

\maketitle
\section{Introduction}
Rademacher's Theorem, a classical theorem in geometric measure theory, states that if $f:\R^m\to \R^n$ is Lipschitz, then $f$ is differentiable at $\mathscr{L}^m$-almost every point $x\in \R^m$, where $\mathscr{L}^m$ denotes the $m$-dimensional Lebesgue measure on $\R^m$ (see for example \cite[Theorem 7.3]{Mattila}). Significant work over the years has gone into investigating potential converses and extensions to this theorem, for instance one could ask whether there exists a Lebesgue-null set $N\subset \R^m$ such that for every Lipschitz function $f:\R^m\to \R^n$, there exists a point $x\in N$ such that $f$ is differentiable at $x$. An answer to this question has been obtained: such a Lebesgue-null set $N\subset\R^m$ exists if and only if $m>n$; see the work of Preiss and Speight in \cite{PS15} for the proof and for a more complete history of this problem. 

In addition to the classical version of Rademacher's Theorem, there are many notions of ``tangents to sets" which also admit a version of this theorem. If $\Gamma\subset \R^d$ is a Lipschitz curve, meaning that it is the image $f([0,1])$ of a Lipschitz map $f:[0,1]\to\R^d$, then $\Gamma$ admits a tangent line at $\mathcal{H}^1$-almost all of its points. This statement holds for approximate tangent planes (see \cite[Theorem 15.19]{Mattila}), for approximate tangent cones (see \cite[Theorem 3.8]{falc}), and most recently for tangents in the sense of Badger and Lewis \cite[Definition 3.1]{BL} (see \cite[Theorem 1.1]{QSS}). Here and throughout this article, $\mathcal{H}^1$ denotes the $1$-dimensional Hausdorff measure. Furthermore, with Vellis we proved in \cite[Appendix A]{QSS} that in every $\R^d$ with $d\geq 2$, there exists a Lipschitz curve $H\subset \R^d$ containing the origin ${\bf 0}$ such that every possible tangent is attained for $H$ at ${\bf 0}$. The goal of this paper is to expand on the work in this direction with the following result.
\begin{theorem}\label{thm0}
Let $d\geq 2$ be an integer and let $K\subset \R^d$ be compact, uniformly disconnected, and admit a Lipschitz capture. Then there exists a Lipschitz capture $F=f([0,1])$ of $K$ such that for every $x\in K$, $\Psi-\wtan(F,x)=\mathfrak{C}_U(\R^d;{\bf 0})$.
\end{theorem}

Here and throughout, we use the notation $\wtan (X,x)$ to denote the set of \emph{tangents} to the set $X$ at the point $x$, $\Psi-\wtan (X,x)$ to denote the set of \emph{pseudotangents} to the set $X$ at the point $x$, and $\mathfrak{C}_U(\R^d;{\bf 0})$ to denote the collection of all closed subsets of $\R^d$ which contain the origin and have only unbounded components, see Section \ref{prelim} for the details of these sets. In Proposition \ref{prop:alltan}, we prove that as long as $X\subset \R^d$ is a nondegenerate continuum, it holds that $\Psi-\wtan(X,x)\subset \mathfrak{C}_U(\R^d;{\bf 0})$ for each point $x\in X$. Note that in Theorem \ref{thm0}, we obtain a result in terms of \emph{pseudotangents}, not tangents. The difference is articulated precisely in Section \ref{prelim}, but for the time being the reader should notice that Lipschitz curves may behave very differently with respect to pseudotangents as opposed to with respect to tangents. In particular, in Example \ref{ex2} we show that there is a nondegenerate Lipschitz curve where \emph{no point} admits a unique pseudotangent, let alone a unique pseudotangent line, while by contrast it is known that a Lipschitz curve must, $\mathcal{H}^1$-almost everywhere, admit a unique \emph{tangent}, and that the unique tangent must be a line. In light of this, we are led to the following question, to which we (perhaps ambitiously) conjecture that the answer is ``yes".
\begin{question}\label{q1}
For each integer $d\geq 2$, does there exist a Lipschitz curve $\Gamma\subset \R^d$ such that for every point $x\in\Gamma$, $\Psi-\wtan(\Gamma,x)=\mathfrak{C}_U(\R^d;{\bf 0})$?
\end{question}

Our strategy for proving Theorem \ref{thm0} involves constructing a Lipschitz capture $G=g([0,1])$ of the set $K$ such that there is a countable set $\{z_k\}_{k\in\N}\subset G\setminus K$ for which for all $k\in\N$, $\wtan(G,z_k)=\mathfrak{C}_U(\R^d;{\bf 0})$ and such that $K\subset \overline{\{z_k\}}_{k\in\N}$. We then prove as a consequence that all of the points $z\in K$, $\Psi-\wtan(G,z)=\mathfrak{C}_U(\R^d;{\bf 0})$; indeed the proof of this fact does not rely on the point $z$ being in $K$, it depends only on $z$ being contained in the closure $\overline{\{z_k\}}_{k\in\N}$. By this argument, we do obtain a stronger result about sets of points in Lipschitz curves on which every tangent is attained than we found in \cite[Appendix A]{QSS}; we improve from showing that there is a Lipschitz curve with \emph{a} point that admits every possible tangent to showing that there is a Lipschitz curve with \emph{a countable set} of points admitting every possible tangent, and that furthermore this countable set can accumulate to an uncountable set.

\subsection{Acknowledgments}
We would like to thank Jeremy Tyson for many helpful conversations giving this project direction, as well as for his commentary on an early draft. Additionally, we thank Vyron Vellis for suggesting Example \ref{ex1} and its inclusion in the final manuscript.

\section{Preliminaries}\label{prelim}
\begin{definition}
Given a set $K\subset \R^d$, we say that a map $f:[0,1]\to \R^d$ or its image $F=f([0,1])$ is a \emph{Lipschitz capture of} $K$ if $f$ is a Lipschitz function and $K\subset F$. We will frequently use a capital letter such as $F$ to denote the image of a map $f$ denoted using the corresponding lowercase letter.
\end{definition}

Following \cite{BL}, given two non-empty subsets $A,B\subset\R^d$, we define the \emph{excess of $A$ over $B$} to be the quantity $\exc(A,B):=\sup_{a\in A}\inf_{b\in B}|a-b|$. This asymmetric quantity measures how far $A$ is from $B$, but not the reverse. For example, $\exc([0,1],\{0\})=1$ while $\exc(\{0\},[0,1])=0$. Following the notation used in \cite{QSS}, we denote by $\mathfrak{C}(\R^d)$ the collection of non-empty closed subsets of $\R^d$, by $\mathfrak{C}(\R^d;{\bf 0})$ the collection of closed subsets of $\R^d$ which contain the origin ${\bf 0}$, and by $\mathfrak{C}_U(\R^d;{\bf 0})$ the collection of sets in $\mathfrak{C}(\R^d;{\bf 0})$ which have only unbounded components. All three of these collections carry a metrizable topology, called the \emph{Attouch-Wets topology}, which is characterized by the following lemma.
\begin{lemma}[{\cite[Chapter 3]{Beer}}]
A sequence of sets $(X_i)_{i\in\N}\subset \mathfrak{C}(\R^d)$ converges to a set $X\in\mathfrak{C}(\R^d)$ in the Attouch-Wets topology if and only if for every $r>0$,
\[\lim_{i\to \infty}\exc(X_i\cap \overline{B}({\bf 0},r),X)=0\text{ and }\lim_{i\to\infty}\exc(X\cap\overline{B}({\bf 0},r),X_i)=0.\]
Furthermore, $\mathfrak{C}(\R^d;{\bf 0})$ is sequentially compact with respect to this topology.
\end{lemma}

\begin{definition}[Tangents, Pseudotangents, {\cite[Definition 3.1]{BL}}]
Given a set $K\in\mathfrak{C}(\R^d)$ and a point $x\in K$, we call a set $T\in\mathfrak{C}(\R^d)$ a \emph{tangent to $K$ at $x$} if there exists a sequence of positive scales $(r_i)_{i\in\N}$ tending to $0$ such that $\lim_{i\to\infty} (r_i)^{-1}(K-x)=T$, where this limit is taken with respect to the Attouch-Wets topology. We call a set $S\in\mathfrak{C}(\R^d)$ a \emph{pseudotangent to $K$ at $x$} if there exists a sequence of positive scales $(r_i)_{i\in\N}$ tending to $0$ and a sequence of points $(x_i)_{i\in\N}\subset K$ with $\lim_{i\to\infty} x_i=x$ such that $\lim_{i\to\infty} (r_i)^{-1}(K-x_i)=S$, where this limit is taken with respect to the Attouch-Wets topology. We denote by $\wtan(K,x)$ the collection of tangents to $K$ at $x$, and by $\Psi-\wtan (K,x)$ the collection of pseudotangents to $K$ at $x$.
\end{definition}

\begin{rem}\label{tanrem}
Note the following two facts about tangents and pseudotangents, the proofs of which are left to the interested reader. 
\begin{enumerate}
\item For any non-empty closed set $K$ and any point $x\in K$, $\wtan(K,x)\subset\Psi-\wtan (K,x)$.
\item If $(x_i)_{i\in\N}\subset K$ converges to a point $x\in K$ and if $T\in\wtan(K,x_i)$ for each $i$, then $T\in\Psi-\wtan(K,x)$ as well.
\end{enumerate}
\end{rem}
Throughout this article, $H$ will denote the set of the same name as in \cite[Appendix A]{QSS}, to which we refer the reader for the details of its construction. The critical properties of $H$ for this article are that $H$ is a finite-length continuum, so by \cite[Theorem 4.4]{AO} it is a Lipschitz curve, which contains ${\bf 0}$, is contained in $[-1,1]^d$, and for which $\wtan(H,{\bf 0})=\mathfrak{C}_U(\R^d;{\bf 0})$.
\section{Results}\label{results}

This section is devoted to the proof of Theorem \ref{thm0}. Note that for a compact, uniformly disconnected set $K\subset \R^d$ admitting a Lipschitz capture as in Theorem \ref{thm0}, it is straightforward to show that there exists a set $\tilde{K}$ containing $K$ such that $\tilde{K}$ admits a Lipschitz capture, is compact, is uniformly disconnected, and is perfect.  Thus we may assume without loss of generality from here on that $K$ itself is perfect in addition to its other properties.

Further, note that the following result holds, with proof exactly the same as the proof of \cite[Lemma 2.5]{QSS}, up to replacing instances of $x$ with $x_n$ instead (to move from tangents to pseudotangents). Structurally the proofs are identical and all of the arguments still hold \emph{mutatis mutandis}.
\begin{proposition}\label{prop:alltan}
Let $d\geq 2$ be an integer and let $X\subset \R^d$ be a nondegenerate continuum. Then for every point $x\in X$, we have that $\Psi-\wtan(X,x)\subset \mathfrak{C}_U(\R^d;{\bf 0})$.
\end{proposition}
This means that the pseudotangent collection given in Theorem \ref{thm0} is maximal at all of the points of $K$. The proof of Theorem \ref{thm0} follows after several lemmata about the geometry of Lipschitz captures of the set $K$. Fix, for the remainder of this section, a compact, uniformly disconnected, perfect set $K\subset\R^d$ which admits a Lipschitz capture $G_0=g_0([0,1])$. By uniform disconnectedness, let $\lambda>0$ such that for every pair of distinct points $x,y\in K$, if $\{p_0,p_1,\dots,p_n\}\subset K$ with $p_0=x$ and $p_n=y$, then there exists some $i\in\{0,1,\dots,n-1\}$ for which $|p_{i+1}-p_i|>\lambda|x-y|$.

\begin{lemma}\label{lem1}
For every pair of distinct $x,y\in K$, for each $a\in g_0^{-1}(\{x\})$ and $b\in g_0^{-1}(\{y\})$, there exist $s,t\in g_0^{-1}(K)$ such that $(s,t)\subset (\min(\{a,b\}),\max(\{a,b\}))\setminus g_0^{-1}(K)$ and such that there is a value $\zeta\in(s,t)$ with $\dist(g_0(\zeta),K)> \frac{1}{4}\lambda|x-y|$, further satisfying that $g_0(\zeta_n)$ is a point of $\mathcal{H}^1$ density $1$ in $G_0$.
\end{lemma}
\begin{proof}
Note that $(\min(\{a,b\}),\max(\{a,b\}))\setminus g_0^{-1}(K)$ is a countable disjoint union of open intervals $(s_i,t_i)$. Since $K$ is uniformly disconnected, there must be some index $i$ for which $|g_0(s_i)-g_0(t_i)|>\lambda |x-y|$, and since $g_0$ is Lipschitz we have also that $|g_0(s_i)-g_0(t_i)|\leq \Lip(g_0)|s_i-t_i|$, so we obtain $|s_i-t_i|>\frac{\lambda}{\Lip(g_0)}|x-y|$. Call these values $s$ and $t$ instead of $s_i$ and $t_i$, and note that since $K$ is not a singleton, $\Lip(g_0)\neq 0$. 

Now if every value $\zeta\in (s,t)$ has $\dist(g_0(\zeta),K)< \frac{1}{3}\lambda|x-y|$, then let $\{\zeta_0,\zeta_1,\dots,\zeta_m\}\subset [s,t]$ be indexed in increasing order and satisfy $\zeta_0=s$, $\zeta_m=t$, and $|\zeta_{i+1}-\zeta_i|\leq\frac{1}{6\Lip(g_0)}\lambda|x-y|$ for each $i\in\{0,1,\dots,m-1\}$. Then there exists a set of points $\{q_0,q_1,\dots,q_m\}\subset K$ with $|q_i-g_0(\zeta_i)|<\frac{1}{3}\lambda|x-y|$ for each $i$, therefore for each $i\in\{0,1,\dots,m-1\}$ we have
\begin{align*}
|q_{i+1}-q_i|&\leq |q_{i+1}-g_0(\zeta_{i+1})|+|g_0(\zeta_{i+1})-g_0(\zeta_i)|+|g_0(\zeta_i)-q_i|\\
&\leq \frac{1}{3}\lambda|x-y|+\frac{1}{6}\lambda|x-y|+\frac{1}{3}\lambda|x-y|<\lambda|x-y|,
\end{align*}
but this contradicts the fact that $K$ is uniformly disconnected if it can be done for every $(s_i,t_i)$ with $|g_0(s_i)-g_0(t_i)|>\lambda |x-y|$ as in the first paragraph. To complete the proof, note that points of $\mathcal{H}^1$ density $1$ form a set of full $\mathcal{H}^1$ measure in $G_0$, so we can perturb $\zeta_n$ slightly to yield the desired inequality and further have that $g_0(\zeta_n)$ is a point of $\mathcal{H}^1$ density $1$ in $G_0$ (see \cite[Theorem 16.2]{Mattila}.
\end{proof}

\begin{lemma}\label{lem2}
For every $x\in K$, there exists a Lipschitz capture $F=f([0,1])$ of $K$ such that
\begin{equation}\label{eqn2}
\Psi-\wtan(F,x)=\mathfrak{C}_U(\R^d;{\bf 0}).
\end{equation}
Moreover, if $G=g([0,1])$ is any Lipschitz capture of $K$ and $r,\delta>0$ are any positive numbers, then we can find a Lipschitz capture $F=f([0,1])$ of $K$ satisfying (\ref{eqn2}) and further satisfying $F\setminus B(x,r)=G\setminus B(x,r)$ and $\mathcal{H}^1(F)-\mathcal{H}^1(G)< \delta$.
\end{lemma}
\begin{proof}
Fix a point $x\in K$ and let $(y_n)_{n\in\N}\subset K\setminus\{x\}$ be a sequence of (distinct) points with $\lim_{n\to\infty} y_n=x$ and with $\sum_{n\in\N}|x-y_n|<\infty$. Let $G=g([0,1])$ be a Lipschitz capture of $K$. For each $n\in\N$, note that $\dist(g^{-1}(\{x\}),g^{-1}(\{y_n\}))=:\eta_n>0$ with $\eta_n\to 0$ as $n\to\infty$, so let $a_n\in g^{-1}(\{x\})$ and $b_n\in g^{-1}(\{y_n\})$ realize $\eta_n=|a_n-b_n|$. Then by Lemma \ref{lem1}, for each $n\in\N$, there exists a value $\zeta_n\in(\min(\{a_n,b_n\}),\max(\{a_n,b_n\}))$ with $\dist(g(\zeta_n),K)> \frac{1}{4}\lambda|x-y_n|$ and such that $g(\zeta_n)$ is a point of $\mathcal{H}^1$ density $1$ in $G$.

We claim now that for each $N\in\N$, there are only finitely many $m\in\N$ such that
\begin{equation}\label{eqn1}
B\left(g(\zeta_{m}),\frac{1}{16}\lambda |x-y_{m}|\right)\cap B\left(g(\zeta_N),\frac{1}{16}\lambda |x-y_N|\right)\neq\emptyset.
\end{equation}
To see this, fix $N\in\N$ and suppose for the sake of contradiction that there is a strictly increasing sequence $(m_k)_{k\in\N}\subset \N$ such that (\ref{eqn1}) holds for each $m_k$. Then for each $k\in\N$, we have that $|g(\zeta_{m_k})-g(\zeta_N)|\leq \frac{1}{16}\lambda (|x-y_{m_k}|+|x-y_N|)$. Note that $\lim_{k\to\infty}g(\zeta_{m_k})=x$ since $|g(\zeta_{m_k})-x|\leq \eta_{m_k}\Lip(g)$ and $\eta_{m_k}\to 0$ as $k\to\infty$. Therefore,
\begin{align*}
\frac{1}{4}\lambda|x-y_N|&\leq |g(\zeta_N)-x|\leq\lim_{k\to\infty}\left(|g(\zeta_N)-g(\zeta_{m_k})|+|g(\zeta_{m_k})-x|\right)\\
&=\lim_{k\to\infty}|g(\zeta_N)-g(\zeta_{m_k})|\leq \lim_{k\to\infty}\frac{1}{16}\lambda(|x-y_{m_k}|+|x-y_N|)\\
&\leq \frac{1}{16}\lambda |x-y_N|,
\end{align*}
but this is impossible as $\lambda|x-y_N|>0$.

Thus, there exists a subsequence $(y_{n_k})_{k\in\N}$ of $(y_n)_{n\in\N}$ for which (\ref{eqn1}) fails for every distinct pair of $k_1,k_2$. For ease of notation, we call this sequence $(y_n)_{n\in\N}$, and similarly relabel $\zeta_n$, $a_n$, and $b_n$. Then define, for $N\in\N$, the sets
\[Z_N:=\left(G\setminus\left(\bigcup_{n\geq N} B\left(g(\zeta_n),\frac{1}{16}\lambda |x-y_n|\right)\right)\right)\cup\left(\bigcup_{n\geq N} \left(\frac{1}{32\sqrt{d}}\lambda|x-y_n| H+g(\zeta_n)\right)\right).\]
Observe also the following inequality:
\[\mathcal{H}^1(Z_N)\leq\mathcal{H}^1(G)+\frac{1}{32\sqrt{d}}\lambda\mathcal{H}^1(H)\sum_{n=N}^\infty |x-y_n|<\infty.\]

To continue with the proof of Lemma \ref{lem2}, we must construct a Lipschitz capture of $Z_N$ which is ``appropriately close to" the existing Lipschitz capture $G$ of $K$. This is made more precise and subsequently proven in the following lemma.
\begin{lemma}\label{guh}
There exists a Lipschitz capture $F_N=f_N([0,1])$ of $Z_N$ satisfying the following three properties for $n\geq N$:
\begin{enumerate}
\item[(i)]
\[F\setminus\left(\bigcup_{n\geq N}B\left(g(\zeta_n),\frac{1}{16}\lambda |x-y_n|\right)\right)=G\setminus\left(\bigcup_{n\geq N}B\left(g(\zeta_n),\frac{1}{16}\lambda|x-y_n|\right)\right),\]
\item[(ii)]
\[F\cap \left(\frac{1}{32\sqrt{d}}\lambda|x-y_n|[-1,1]^d+g(\zeta_n)\right)=H_n,\]
\item[(iii)]
$\mathcal{H}^1(F)-\mathcal{H}^1(G)$ can be made smaller than an arbitrary positive number,
\end{enumerate}
where $H_n:=\frac{1}{32\sqrt{d}}\lambda|x-y_n|H+g(\zeta_n)$.
\end{lemma}
\begin{proof}[Proof of Lemma \ref{guh}]
Fix $N\in\N$ and $n\geq N$. Choose the values $\zeta_n$ as in Lemma \ref{lem1}. Note that $g^{-1}(G\cap B(g(\zeta_n),\frac{1}{12} \lambda |x-y_n|))$ is the disjoint union of countably many open intervals $I^n_k=(s^n_k,t^n_k)$. Enumerate these intervals so that $\zeta_n\in I^n_1$. For $k\neq 1$, let $J^n_k\subset \partial B(g(\zeta_n),\frac{1}{16} \lambda |x-y_n|)$ be a geodesic on this sphere connecting $g(s^n_k)$ to $g(t^n_k)$. Note that $\mathcal{H}^1(J^n_k)\leq \frac{\pi}{2} \mathcal{H}^1(g(I^n_k))$.

For $k=1$, let 
\[p_1,p_2\in\left(\frac{1}{32\sqrt{d}}\lambda |x-y_n|(\partial [-1,1]^d\cap H)+g(\zeta_n)\right)\]
be, respectively, a point closest to $g(s^n_1)$ and to $g(t^n_1)$ in the latter set. Let \[J^n_1:=[g(s^n_1),p_1]\cup[g(t^n_1),p_2]\cup H_n,\]
where $[g(s^n_1),p_1]$ and $[g(t^n_1),p_2]$ are the line segments connecting their respective endpoints. Then note that $\mathcal{H}^1(J^n_1)\leq \lambda |x-y_n|$ and that $\mathcal{H}^1(g(I^n_1))\geq \frac{1}{2} \lambda |x-y_n|$. Let $F_n:=Z_N\cup(\cup_k J^n_k)$. Observe the following:
\begin{align*}
\mathcal{H}^1(F_n)-\mathcal{H}^1(G)&\leq\sum_k (\mathcal{H}^1(J^n_k)-\mathcal{H}^1(g(I^n_k)))\\
&\leq \frac{1}{2}\lambda |x-y_n|+(\frac{\pi}{2}-1)\sum_{k\geq 2}\mathcal{H}^1(g(I^n_k))\\
&\leq \frac{1}{2}\lambda |x-y_n|+(\frac{\pi}{2}-1)\mathcal{H}^1(G\cap B(g(\zeta_n),\frac{1}{16}\lambda |x-y_n|))\\
&\leq C_0|x-y_n|,
\end{align*}
where the constant $C_0$ in the last line depends only on $\lambda$, a parameter of the geometry of $K$, and on $g(\zeta_n)$ being a point of $\mathcal{H}^1$ density $1$ in $G$. In particular, it does not depend on $n$ or on $N$.

By the disjointness assumption from the failure of (\ref{eqn1}) for $(y_{n_k})$, letting $\overline{F}_N=\cup_{n\geq N}F_n$, we have that $\overline{F}_N$ has property (i) and has property (ii) for $n\geq N$. Furthermore, $\mathcal{H}^1(\overline{F}_N)-\mathcal{H}^1(G)\leq C\sum_{n\geq N}|x-y_n|$, where the constant $C$ depends only on $\lambda$, and does not depend on $N$. Then since $\sum_{n\in\N} |x-y_n|<\infty$, we can make this difference arbitrarily small, yielding property (iii). This concludes the proof of Lemma \ref{guh}.
\end{proof}
To conclude the proof of Lemma \ref{lem2}, the only remaining observation to make is that $F_N$ is a finite-length continuum containing $Z_N$ and satisfying (\ref{eqn2}). Alongside properties (i)-(iii), this implies that $F_N$ is a Lipschitz capture of $K$ with the desired properties. Furthermore, by the failure of (\ref{eqn1}), by the fact that $\wtan(H,{\bf 0})=\mathfrak{C}_U(\R^d;{\bf 0})$, by property (ii) of Lemma \ref{guh}, and by the second part of Remark \ref{tanrem}, we have (\ref{eqn2}) is satisfied by $F_N$ since $y_n\to x$ as $n\to\infty$.
\end{proof}

\begin{proof}[Proof of Theorem \ref{thm0}]
We are finally ready to prove Theorem \ref{thm0}. Let $\{x_n\}_{n\in\N}$ be a countable  dense subset of $K$ and let $G_1=g_1([0,1])$ be a Lipschitz capture of $K$ as in Lemma \ref{guh} with $x=x_1$, in particular, satisfying (\ref{eqn2}) with $x=x_1$. Let $(\delta_n)_{n\in\N}>0$ such that $\sum_{n\in\N} \delta_n<\infty$. We proceed to construct the desired Lipschitz capture $G=g([0,1])$ of $K$ recursively.

Let $n\in\N$ and assume we have already a Lipschitz capture $G_n=g_n([0,1])$ of $K$ such that for each $k\in\{1,\dots,n\}$, $\Psi-\wtan (G_n,x_k)=\mathfrak{C}_U(\R^d;{\bf 0})$. 

We now construct an appropriate Lipschitz capture $G_{n+1}=g_{n+1}([0,1])$ of $K$. Let $r_{n+1}>0$ be small enough that for each $k\in\{1,\dots,n\}$, $x_k\notin \overline{B}(x_{n+1},2 r_{n+1})$. By Lemma \ref{guh}, let $G_{n+1}=g_{n+1}([0,1])$ be a Lipschitz capture of $K$ such that $G_{n+1}\setminus B(x_{n+1},r_{n+1})=G_n \setminus B(x_{n+1},r_{n+1})$ and $\mathcal{H}^1(G_{n+1})-\mathcal{H}^1(G_n)<\delta_{n+1}$. Then for every $k\in\{1,\dots,n\}$, $\Psi-\wtan (G_{n+1},x_k)=\mathfrak{C}_U(\R^d;{\bf 0})$ and by construction, $\Psi-\wtan (G_{n+1},x_{n+1})=\mathfrak{C}_U(\R^d;{\bf 0})$ as well. This follows from a sequence of simple observations.

First, observe that the Hausdorff limit $\lim_{n\to\infty} G_n=:G$ exists, $K\subset G$, and $G$ is a continuum of finite length by Go{\l}ab's Semicontinuity Theorem, so it is a Lipschitz capture of $K$ by \cite[Theorem 4.4]{AO} (see this same paper for discussion of Go{\l}ab's Semicontinuity Theorem as well). We now claim that for every $n\in\N$, $\Psi-\wtan (G,x_n)=\mathfrak{C}_U(\R^d;{\bf 0})$.

To see this, note that when $G_{n+1}$ is constructed from $G_n$ as in Lemma \ref{guh}, the sets described in property (ii) of this lemma for $G_n$ do not intersect $B(x_{n+1},r_{n+1})$. Therefore, (\ref{eqn2}) still holds for $G_{n+1}$ and for $G_n$ about $x_n$. That is, for every $k\in\{1,\dots,n+1\}$, $\Psi-\wtan(G_{n+1},x_k)=\mathfrak{C}_U(\R^d;{\bf 0})$. Then the same holds for subsequent $G_{n+m}$ constructed in this manner, thus we have that for every $n\in\N$, $\Psi-\wtan(G,x_n)=\mathfrak{C}_U(\R^d;{\bf 0})$.

Finally, by the second part of Remark \ref{tanrem}, we have that $\Psi-\wtan(G,x)=\mathfrak{C}_U(\R^d;{\bf 0})$ for every $x\in K$, as desired.
\end{proof}

\begin{rem}
Note that if $K\subset \R^d$ is any compact set with Assouad dimension $\dim_A(K)<1$, then $K$ is also uniformly disconnected and admits a Lipschitz capture, so every compact set with Assouad dimension less than $1$ admits a Lipschitz capture which attains every possible pseudotangent at every point of $K$.
\end{rem}

\section{Examples}\label{conclusion}

In the following two examples, we show that Theorem \ref{thm0} is not an exhaustion of the topic of tangents or pseudotangents to a Lipschitz curve. The assumptions on $K$ are not necessarily strict. In particular, Question \ref{q1} is a very strong question one could ask about pseudotangents to Lipschitz curves, and it is still open. 
\begin{ex}\label{ex1}[A set of Hausdorff dimension $1$ which attains every pseudotangent along a Lipschitz curve]

For natural numbers $k\geq 2$, let $C_k\subset [0,1/k^2]$ be a self-similar Cantor set of Hausdorff dimension $\log(k)/\log(k+1)$. Then by self-similarity, $\dim_A(C_k)=\log(k)/\log(k+1)$ as well. In $\R^d$, for $d\geq 2$, let
\[K:=\{{\bf 0}\}\cup\left(\bigcup_{k\geq 2} \{0\}^{d-2}\times \{1/k\}\times C_k\right).\]
Then $K\subset \R^d$ is compact, $\dimh(K)=1$, $\mathcal{H}^1(K)=0$, and $K$ admits a Lipschitz capture $F$ by applying Theorem \ref{thm0} to each $C_k$ in such a way that the curves do not intersect and such that the sum of the lengths of the curves is finite, then adding also the segment $[0,1]\times\{0\}$. Then for all $x\in K$, $\Psi-\wtan (F,x)=\mathfrak{C}_U(\R^d;{\bf 0})$.

\end{ex}

\begin{ex}\label{ex2}[A Lipschitz curve with non-unique pseudotangents at every point]

Let $\{q_j\}_{j\in\N}$ be an enumeration of $\Q\cap [0,1]$, let $K_0:=[0,1]\times\{0\}\subset \R^2$, and let \[K_1:=K_0\cup\left(\bigcup_{j\in\N}\{q_j\}\times [0,\frac{1}{4(j+1)^2}]\right).\]
We proceed by recursion.

For $n\geq 1$, given finite-length continua $K_{n-1}\subset K_n$, satisfying that the components of $K_n\setminus K_{n-1}$ are either all parallel to the $x$-axis or all parallel to the $y$-axis, and such that $K_{n-1}\subset \overline{K_n\setminus K_{n-1}}$, we construct $K_{n+1}$ as follows. Let $\{(x_j^n,y_j^n)\}_{j\in\N}$ be a countable dense subset of $K_n\setminus K_{n-1}$. If the components of $K_n\setminus K_{n-1}$ are all parallel to the $x$-axis, then let
\[K_{n+1}:=K_n\cup\left(\bigcup_{j\in\N} \{x_j^n\}\times[y_j^n,y_j^n+\frac{1}{(j+1)^2(n+1)^2}]\right),\]
and if the components of $K_n\setminus K_{n-1}$ are instead all parallel to the $y$-axis, then let
\[K_{n+1}:=K_n\cup \left(\bigcup_{j\in\N}[x_j^n,x_j^n+\frac{1}{(j+1)^2(n+1)^2}]\times\{y_n^j\}\right).\]

The Hausdorff limit $K_\infty$ of these sets $K_n$ exists and by  \cite[Theorem 4.4]{AO} and by Go{\l}ab's Semicontinuity Theorem, it is a continuum of finite length, and so a Lipschitz curve. It is straightforward to see that since points with a unique tangent line are dense in $K_\infty$ and since points admitting a tangent which is not a line are also dense in $K_\infty$, every point of $K_\infty$ admits at least two distinct pseudotangents.
\end{ex}
\bibliography{bib}
\bibliographystyle{amsbeta}

\end{document}